\documentclass[a4paper,11pt]{article} 
\usepackage[utf8]{inputenc} 

\usepackage{amsmath}
\usepackage{amssymb}
\usepackage{times}
\usepackage{todonotes}

\usepackage{geometry} 
\usepackage{graphicx} 

\usepackage{booktabs} 
\usepackage{array} 
\usepackage{paralist} 
\usepackage{verbatim} 
\usepackage{subfig} 

\usepackage{fancyhdr} 
\usepackage{amsthm}
\usepackage{thmtools}
\declaretheorem[name=Theorem,numberwithin=section]{thm}
\declaretheorem[name=Lemma,sibling=thm]{lemma}

\declaretheorem[name=Proposition,sibling=thm]{prop}
\declaretheorem[name=Corollary,sibling=thm]{cor}
\declaretheorem[name=Definition, style=definition,sibling=thm]{defn}
\declaretheorem[name=Remark,style=remark,sibling=thm]{rmk}

\usepackage{nameref,hyperref,cleveref}

\usepackage{lipsum}
\usepackage{enumitem}
\setlist[enumerate]{itemsep=0mm}

\usepackage[nottoc,notlof,notlot]{tocbibind} 
\usepackage[titles,subfigure]{tocloft} 

\newcommand{\bvar}{\bar{\varphi}}
\newcommand*\diff{\mathop{}\!\mathrm{d}}   
\newcommand{\dt}{d_{\theta}}  
\newcommand{\dx}{d_{\X}} 
\newcommand{\ess}{\operatorname{ess}}
\newcommand{\dxa}{\dx^{\alpha}}
\newcommand{\dist}{\bigl(\dt(\om,\tom)+\dx^\alpha(x,y)\bigr)}
\newcommand{\mE}{\mathbb{E}}  
\newcommand{\F}{\mathcal{F}} 

\newcommand{\Glm}{\Gamma(l_{\min})}
\newcommand{\gta}{G_{\theta,\alpha}^+} 
\newcommand{\gtan}{(\gta,\|\cdot\|_\ta)}
\newcommand{\HH}{H}
\newcommand{\id}{\operatorname{id}}  
\newcommand{\imag}{\mathrm{i}}  
\newcommand{\im}{\operatorname{im}}  
\newcommand{\K}{\mathcal{K}}  
\newcommand{\cL}{\mathcal{L}}  

\newcommand{\mat}{GL(d,\mathbb{R})}
 
\newcommand{\nn}{\nu_{\sigma^{-n}\om}}
\newcommand{\op}{\Omega^+} 
\newcommand{\opx}{\Omega^+\times\X}
\newcommand{\om}{\omega}
\newcommand{\tom}{\tilde{\omega}}
\newcommand{\omg}[1]{\omega^{(#1)}}
\newcommand{\omn}{\omega^{(n)}}
\newcommand{\tomn}{\tom^{(n)}}

\newcommand{\mP}{\mathbb{P}} 
\newcommand{\cP}{\mathcal{P}} 
\newcommand{\pn}{\psi(n,\sigma^{-n}\om)}
\newcommand{\vn}{\varphi(n,\sigma^{-n}\om)}
\newcommand{\R}{\mathbb{R}}
\newcommand{\proj}[1]{\mathbb{R}P^{#1}}  
\newcommand{\real}{\operatorname{Re}} 
\newcommand{\cS}{\mathcal{S}}  
\newcommand{\supp}{\operatorname{supp}}   
\newcommand{\shft}{(\op,\mathcal{F}^+,\mP)}
\newcommand{\SPhi}{\Sigma}
\newcommand{\SM}{\mathfrak{M}}
\newcommand{\ta}{{\theta,\alpha}}
\newcommand{\tnu}{\tilde{\nu}}
\newcommand{\tvar}{\tilde{\varphi}}
\newcommand{\ovW}{\overline{W}}
\newcommand{\X}{\mathcal{X}}

\newcommand{\nso}{{\sigma^n\omega^{(n)}=\omega}}
\newcommand{\oso}{{\sigma\omega'=\omega}}
\newcommand{\tnso}{{\sigma^n\tom^{(n)}=\tom}}

\newcommand\numberthis{\addtocounter{equation}{1}\tag{\theequation}}

\title{On Transfer Operators for Markovian Products of Invertible Random Matrices}
\author{Fan Wang\thanks{Department of Statistics, University of Oxford. Email: \href{mailto:wangfan.ox@gmail.com}{wangfan.ox@gmail.com}}  \and David Steinsaltz\thanks{Department of Statistics, University of Oxford. Email: \href{mailto:steinsal@stats.ox.ac.uk}{steinsal@stats.ox.ac.uk} }}

\begin{document}
\maketitle

\section{Introduction}
\label{sec:intro}
The Lyapunov exponent for products of random matrices plays a similar foundational role in the asymptotic theory of non-commutative random variables to that played by the Law of Large Numbers in classical probability theory. Consequently,
methods for computing Lyapunov are of widespread interest,
and in recent years attention has focused on methods
established in dynamical systems.

Transfer operators, which are usually symbolically defined to be the ``adjoint'' of the pullback of a shift operator, are also essential and powerful in multiple areas in probability theory, such as classical Markov chain theory \cite{Nagaev},  stochastic differential equations, etc. They are also known as Ruelle--Perron--Frobenius operators and can be used to study the periodic points in symbolic dynamics via zeta functions \cite{PP},  
The applications of transfer operator to products of random matrices was introduced by 
 \cite{Tutubalin,clt}, where 
 the derivative of a perturbed transfer operator was shown to be
 equal to the Lyapunov exponent. This method can be further
 used to obtain an analogue of the Central Limit Theorem  \cite{Bougerol}. In a recent paper \cite{Pollicott} Pollicott also used the nuclearity of certain types of transfer operators (which was first given by \cite{Ruelle}) to provide an efficient algorithm for computing the Lyapunov exponent for i.i.d.\ products of positive matrices.

Research attention, whether in computing the Lyapunov exponent or applying the technique of transfer operators to study products of random matrices, has been devoted primarily to i.i.d.\ products of random matrices. The more general results for Markov chains
of random matrices have been neglected, and sometimes erroneously
assumed to be entirely straightforward.

The difficulty in the generalisation arises from the essential difference between an i.i.d.\ shift and a Markovian shift: In the Markovian setting
the shift needs to store the current state in its ``cache''.

In \cite{MPPM} we focus on the Markovian products of positive matrices. We resolve the problem by generalizing the transfer operator to a matrix of transfer operators, and show that this larger construct has the necessary spectral properties. This leads to a complete generalization of the algorithm proposed by Pollicott.

In the present article we address the case of Markovian products of invertible (not necessarily positive) matrices chosen from a strongly irreducible, contracting, finite set of matrices. Without positivity a simple formula like that of Pollicott does not hold, but we show that we can reconstruct the following crucial elements underlying the link between the transfer operators and the Lyapunov exponent:

\begin{enumerate}[label=(\alph*)]
\item proof that the invariant measure is proper (definition is given by \autoref{defn:properness}) --- see \autoref{prop:proper.LRDS};
\item construction of transfer operators in the language of random dynamical systems --- see \eqref{eq:defn.L.mkv};
\item construction of an appropriate function space where the transfer operators act --- see \Cref{sec:gta};
\item proof of a Lasota--Yorke inequality --- see \autoref{thm:basic.ineq.Lt}.
\end{enumerate}

We follow the path of \cite{Bougerol} to give Markovian generalisations of the classical results on i.i.d.\ random matrix products in \Cref{sec:mkv}. And then we apply generalised Markovian transfer operators to Markovian products of matrices to obtain the spectral theorem in \Cref{sec:to}.
The choice of the function space (on which transfer operators will be acting) is given in \Cref{sec:gta}; the definition of a Markovian transfer operator can be found in \Cref{sec:mkv.to}; and finally the spectral theorems are proved in \Cref{sec:spec}.

We have tried to make clear in our
proofs where the Markov property is being used, as some of the arguments carry over equally
--- or with additional assumptions weaker than Markov --- to more general stationary shifts. 
The most significant application of the Markov property is to show that
the stationary distribution factorises into a finite convex combination
of conditional distributions, corresponding
to the $k$ different possible values of the most recent matrix
in the product. This is required to show that the stationary shift is
{\em proper}, meaning that it puts zero probability onto proper
linear subspaces of $\R^d$ --- Propositions \ref{prop:proper.LRDS} and \ref{prop:limit.range} --- and consequently the fundamental
identity --- Theorem \ref{thm:Furstenberg23.LRDS} --- of the Lyapunov exponent as then mean log growth
of a random vector selected from the stationary distribution.
If we know this to hold for some other reason then we may apply
our results more generally.

\section{Notation and Definitions}
\label{sec:notations}
\begin{defn}
	Let $E$ be a measure space and $\mathbb{T}=\mathbb{Z}$ or $\mathbb{N}$, then $(\Omega,\F)$ is said to be a \emph{shift space} over $E$ if
	\[
	\Omega=E^\mathbb{T}=\left\{\omega=(\omega_i)_{i\in\mathbb{T}}: \om_i\in E\right\},
	\]
	equipped with the sigma-algebra $\F$ generated by cylinder sets
	\[\{\omega\in\Omega: \om_{i_1}=j_1, \dots, \om_{i_m}=j_m \},\quad  i_1,\dots, i_m\in\mathbb{T},  j_1\dots, j_m\in E,\]
	the shift map $\sigma:\Omega\to\Omega$ satisfying $(\sigma\om)_i=\om_{i+1}$ for each $i\in\mathbb{T}$, and an invariant measure  $\mP$ with respect to $\sigma$.  
	
	A shift space $(\Omega,\F,\mP,\sigma)$ is said to be a \emph{Markovian shift} if $\mP$ is Markovian. It is called a \emph{full shift} if $\mP$ associates positive probability to any non-empty cylinder sets. $(\Omega, \F,\mP,\sigma)$ is said to be a \emph{shift over $k$ symbols} if $E=\{1,2,\dots, k\}$. (Unless otherwise indicated, \emph{we always assume the shift space is full}.)
	
	In particular, when $\mathbb{T}=\mathbb{Z}$, we say the shift space is \emph{two-sided}; and when $\mathbb{T}=\mathbb{N}$, say it is \emph{one-sided}. A one-sided shift space is usually denoted by $(\op,\F^+,\mP,\sigma)$.
\end{defn}

With a fixed measure space $E$ we obtain a one-sided shift space $\op$ from a two-sided $\Omega$ by restricting $\mathbb{Z}$ to $\mathbb{N}$, i.e., 
\begin{equation}\label{eq:r}
r^+:\Omega\to\op,\quad r^+(\omega)=(\om_0,\om_1,\dots).
\end{equation}

With a fixed measure space $E$ we obtain a one-sided shift space $\op$ from a two-sided $\Omega$ by restricting $\mathbb{Z}$ to $\mathbb{N}$, i.e., 
\begin{equation}\label{eq:r}
r^+:\Omega\to\op,\quad r^+(\omega)=(\om_0,\om_1,\dots).
\end{equation}

Denote by $\Omega^+$ the associated one-sided shift space of $\Omega$ by ``throwing away the past''.

In this article, unless otherwise indicated, \emph{we alway assume the shift space is full}.

\begin{defn}
Let $X$ be a measure space, $\Omega$ be a shift space. Then the measurable function
\[
	\varphi: \mathbb{T}\times\Omega\times X \to X, \quad (n,\om, x)\mapsto\varphi(n,\om,x)=:\varphi(n,\om)x
\] 
is said to be a \emph{discrete-time random dynamical system (RDS)} acting on $X$ if 
\begin{enumerate}[label=(\alph*)]
\item $\varphi(0,\om)=\id_X$ for any $\om\in\Omega$;
\item $\varphi(m+n,\om)=\varphi(n,\sigma^m\om)\circ\varphi(m,\om)$ for any $m,n\in\mathbb{T}, \om\in\Omega$.
\end{enumerate}
When $\varphi$ is an RDS acting on $\mathbb{R}^d$, after fixing a basis we can assume $\varphi(n,\om)$ is a $d\times d$ matrices for each $(n,\om)\in\mathbb{T}\times\Omega$. We call such a $\varphi$ a \emph{linear random dynamical system (LRDS)}.
\end{defn}

One can easily restrict an RDS $\varphi$ on $\mathbb{Z}\times\Omega$ with two-sided time to $\varphi^+=\varphi|_{\mathbb{N}\times\Omega^+}$ with one-sided time.

\begin{defn}
The \emph{Lyapunov exponent} $\gamma$ associated with an LRDS $\varphi$ is given by
\[ \gamma := \lim_{n\to\infty}\dfrac{1}{n}\mE[\log\|\varphi(n,\om)\|].\]
The limit exists by the sub-additive inequality and does not depend on the choice of the matrix norm $\|\cdot\|$ since matrix norms are equivalent.
\end{defn}

\begin{defn}
If an RDS $\varphi$ acting on a measure space $X$, if $\varphi(1,\om)$ is invertible as an endomorphism (in other words, it is an automorphism) of $X$ for each $\om$, we say $\varphi$ is \emph{spatially invertible}.
\end{defn}

Denote by $\proj{d-1}$ the associated projective space of $\mathbb{R}^d$. Then when $\varphi$ is a spatially invertible LRDS acting on $ \mathbb{R}^d $, it induces a well-defined RDS $\bvar$ acting on $\proj{d-1}$ by
\[\bvar(n, \om)[x]=[\varphi(n,\om)x],\quad [x]\in\proj{d-1},\]
where $ [x] $ denotes the equivalent class containing the non-zero vector $ x\in \mathbb{R}^d $.
We call $\bvar$ the \emph{normalized random dynamical system (NRDS)} induced by the spatially invertible LRDS $\varphi$.

\begin{defn}\label{defn:skew.prod}
Given an RDS $\varphi$ acting on a measure space $X$, define the mapping
\[
	\SPhi: (\om, x)\mapsto (\sigma\om, \varphi(1,\om)x)
\]
from $\Omega^+\times X$ to itself, called the \emph{skew product} of the shift space $\op$ and the RDS $\varphi$. 

An \emph{invariant measure} $\nu$ for the RDS $\varphi$ is defined to be a probability measure on $\Omega\times X$ satisfying $\SPhi\nu=\nu$ and $\pi_{\Omega}\nu=\mP$, where $\pi_\Omega:\Omega\times X\to\Omega$ is the natural projection $\pi_\Omega(\om, x)=\om.$
\end{defn}
Since $ \mathbb{R}P^{d-1} $ is compact, the NRDS $ \bvar $, induced by an LRDS $\varphi$, admits an invariant measure on $\mathbb{R}P^{d-1}$ by  Theorem 1.5.10 of \cite{Arnold}. Moreover, since $\proj{d-1}$ is Polish, by Proposition 1.4.3 of \cite{Arnold} there exists a unique factorisation of the invariant measure. That is,
symbolically we may represent the invariant measure as
\[
	\nu(\diff\om, \diff x)=\nu_\om(\diff x)\mP(\diff\om).
\]

\begin{rmk}\label{rmk:inv.meas.corsp}
Let $\varphi$ be a spatially invertible RDS with one-sided time $\op$. 
We can extend it to an RDS $\tvar$ with two-sided time by defining $\tvar(1,\om)=\varphi(1,r^+(\om))$ and $\tvar(-1,\om)=\varphi(1,r^+(\sigma^{-1}\om))^{-1}$ for $\om\in\Omega$, where $r^+:\Omega\to\op$ is the natural restriction map given by \eqref{eq:r}. By Theorem 1.7.2 of \cite{Arnold} there is a
one-to-one correspondence between the invariant measures for $\varphi$ and those for $\tvar$.
\end{rmk}

\begin{defn}
Let $\varphi$ be a spatially invertible LRDS over the one-sided full shift $\Omega^+$. Denote by $\cS_\varphi$ the semigroup
\[\cS_\varphi=\{\varphi(n,\om): n\ge 0, \om\in\Omega^+ \}.\]
We say $\varphi$ is \emph{strongly irreducible} if there 
is no finite family of proper linear subspace $V_1, \dots, V_k$ of $\mathbb{R}^d$ such that
\[A(V_1\cup\cdots\cup V_k)=V_1\cup\cdots\cup V_k,\]
for each $A\in\cS_\varphi$. 

We say $\varphi$ is \emph{has index $r$} if there exists a sequence $(M_i)_{i\ge 1}\subset \cS_\varphi$ such that $M_i/\|M_i\|$ converges to a matrix of rank $r$ as $i\to\infty$. If $r=1$, we also say $\varphi$ is \emph{contracting}.
\end{defn}
We use throughout the standard notation $f^+(x):=\max\{f(x),0\}$ for any measurable function $f$.

\section{Markovian Products of Matrices}
\label{sec:mkv}
In this section, we generalise the standard results on i.i.d.\ products of random matrices (which can be found in, for example, \cite{Bougerol}) to Markovian matrix products using the language of random dynamical systems. The main goal here is to prove 
\autoref{prop:proper.LRDS}, asserting that the invariant measure with respect to a Markovian RDS is proper (see \autoref{defn:properness}).

\begin{rmk}
	Note that the element $\omega$ of a one-sided Markovian shift
	represents the infinite past of the Markov chain (in the
	usual representation) running off to the right, with $\omega_0$
	the current state. The process advances one step into the
	future by $\SPhi^{-1}$, with the range of possible steps
	corresponding to the non-uniqueness of $\sigma^{-1}$.
	Hence, note that the matrices that define steps of
	the Markov chain on the space $X$
	are the inverses of the matrices that define
	$\varphi(1,\omega)=M(\omega)$ (see \eqref{eq:Phi}).
\end{rmk}

\begin{thm}\label{thm:Furstenberg1.LRDS}
Let $(\Omega^+,\mathcal{F}^+,\mathbb{P},\sigma)$ be a one-sided shift, $\varphi$ be a spatially invertible LRDS. If $\mE[\log^+\|\varphi(1, \omega)\|]<\infty$,  
then the limit
\[
	\lim_{n\to\infty}\dfrac{1}{n}\log\|\varphi(n,\omega)\|
\]
exists for $\mathbb{P}$-almost all $\omega$. Denote the limit by $\gamma(\omega)$. If $\mP$ is ergodic we have that $\gamma(\omega)$ equals the Lyapunov exponent $\gamma$,  $\mathbb{P}$-a.s..
\end{thm}
\begin{proof}
This is an easy consequence of the sub-additive ergodic theorem (see e.g. \cite{Durrett}).
\end{proof}

When studying products of random matrices using the language of RDS, we usually focus on one-sided shift spaces. This creates
some difficulties, as the shift map $\sigma$ is then not invertible. The following proposition (a version of Theorem 1.7.2 of \cite{Arnold}) builds a bridge between the one-sided
and two-sided processes.

\begin{prop}\label{prop:limit.inv.meas}
Let $(\Omega,\mathcal{F},\mathbb{P},\sigma)$ be a two-sided  Markovian full shift over $k$ symbols, $\psi$ be an RDS acting on the Polish space $X$, determined by $\psi(1,\omega)=M(\omega)$, where $M(\om)$  depends only on the coordinate $\om_0$. Let $\nu$ be an invariant measure for $\psi|_{\mathbb{N}\times\op}$, and $\nu(\diff\om, \diff x)=\nu_\om(\diff x)\mP(\diff\om)$ be its unique factorisation. Then
\begin{enumerate}[label=(\alph*)]
	\item \label{it:first_coord}
	$\nu_\om$ also depends only on the coordinate $\om_0$. In other words, there exist $k$ probability measures $\nu_1,\dots, \nu_k$ such that $\nu=\sum_{i=1}^k q_i\nu_i\otimes \delta_{\{\omega:\omega_0=i\}}$, where $q_i=\mP(\om_0=i)$.
\item \label{it:mu_converges} $\mu_{n,\omega} : = \psi(n,\sigma^{-n}\om)\nu_{\sigma^{-n}\om}$ converges weakly to a probability measure $\tnu_\om$ as $n\to\infty$, where $\tnu$ is an invariant measure for $\psi$ with two-sided time;
\item \label{it:Q_converges} Let $Q$ be any finite product $M_{i_m}\cdots M_{i_1} M_{i_0}$. Then
\[
\lim_{n\to\infty}\psi(n,\sigma^{-n}\om)Q \nu_{i_0}=\tnu_\om \, ,
\]
for $P^m$-almost every $Q$ and $\mP$-almost every $\omega$. 
\end{enumerate}
\end{prop}

\begin{proof}
For \ref{it:first_coord}, since $\nu_\om$ is $\F^+$-measurable, $\nu_\om=\mE[\nu_\om|\F^+]=\mE[\nu_\om|\om_0]$, is a function depending only on $\om_0$. Let $\nu_i$ be $\nu_\om$ when $\om_0=i$, then
\[\nu=\int\nu_\om\mP(\diff\om)=\sum_{i=1}^k q_i\nu_i,\]
where $q_i=\mP(\om_0=i)$.

The proof of \ref{it:mu_converges} can be found in \cite{Arnold} Theorem 1.7.2. 
If we define a filtration $\mathcal{F}_{-m}^+=\sigma^{m}\F^+$, then
for any bounded measurable function $f:X\to \R$, $\mu_{n,\om} (f)$
is a bounded martingale with respect to $\{\mathcal{F}_{-m}^+:m\ge0\}$. The result then follows by the
Martingale Convergence Theorem.

Part \ref{it:Q_converges} follows along the same lines as in \cite[Lemma 2.13]{frontiere} or \cite[Lemma II.2.1]{Bougerol}. Because the shift
is full, the statement we wish to prove is equivalent to the claim that
\[
	\lim_{n\to\infty}\psi(n,\sigma^{-n}\om) Q M_{\omega_0}\nu_{\om} =\tnu_\om
\]
for $P^m$-almost every $Q$ and $\mP$-almost every $\omega$. 
And this is equivalent to
\begin{equation} \label{E:omegamn}
	\lim_{n\to\infty}\psi(n,\sigma^{-n-m}\om) \nu_{\omega_{-m-n}} =\tnu_\om
\end{equation}
almost surely in the distribution $\mathbb{Q}_n$ that makes $(\dots,\omega_{-n-2},\omega_{-n-1})$ and $(\omega_{-n},\omega_{-n+1},\dots)$ independent choices from the left-sided and right-sided
restrictions of $\mP$ respectively. Since the shift
is full and Markovian, $\diff \mathbb{Q}_n / \diff \mP$ is bounded
away from 0,
this is equivalent to \eqref{E:omegamn} for $\mP$-almost every $\omega$. By part
\ref{it:mu_converges} this will follow if we show that
\begin{equation}
\label{E:almostsure}
	\sum_{n=1}^\infty \left|\psi(m+n, \sigma^{-m-n}\om) \nu_{\omega_{-m-n}}(f)-\pn\nn(f)\right|^2
\end{equation}
is almost surely finite, and this will follow {\em a fortiori} if the expected value
of \eqref{E:almostsure} is finite.

For any bounded Borel function $f$ on $\op\times X$, 
\begin{align*}
 \mE&\left|\psi(m+n, \sigma^{-m-n}\om) \nu_{\omega_{-m-n}}(f)-\pn\nn(f)\right|^2 \\
&\qquad= \mE\left|\pn\psi(m, \sigma^{-m-n}\om) \nu_{\sigma^{-m-n}\om}(f) \right|^2
+ \mE \left| \pn\nn(f) \right|^2 \\
&\qquad\qquad -2\mE\bigl[\left(\pn\psi(m, \sigma^{-m-n}\om) \nu_{\sigma^{-m-n}\om}(f) \right)\left( \pn\nn(f) \right)\bigr]
\end{align*}
Taking the conditional expectation with respect to $\F^+_{-n}$, and using the fact
that the action of $\psi(n,\sigma^{-n}\omega)$ on a measure is linear, we have
\begin{align*}
	\mE&\bigl[\left(\pn\psi(m, \sigma^{-m-n}\om) \nu_{\sigma^{-m-n}\om}(f) \right)\left( \pn\nn(f) \right)\bigr] \\
	&=\mE\Bigl[\mE\bigl[\left(\pn\psi(m, \sigma^{-m-n}\om) \nu_{\sigma^{-m-n}\om}(f) \right)\left( \pn\nn(f) \right) \, | \, \F_{-n}^+\bigr] \Bigr]\\
	&=\mE\Bigl[\pn\mE\bigl[\left(\psi(m, \sigma^{-m-n}\om) \nu_{\sigma^{-m-n}\om}(f) \right) \, | \, \F_{-n}^+\bigr] \left( \pn\nn(f) \right)\Bigr]\\
	&=\mE\Bigl[\left( \pn\nn(f) \right)^2\Bigr].
\end{align*}
by the invariance of $\nu$. Thus
\begin{align*}
	\sum_{n=1}^\infty &\mE\left|\pn\psi(m, \sigma^{-m-n}\om) \nu_{\omega_{-m-n}}(f)-\pn\nn(f)\right|^2\\
		&\le \sum_{n=1}^\infty \mE\left|\pn\psi(m, \sigma^{-m-n}\om) \nu_{\omega_{-m-n}}(f)\right|^2- \mE \left|\pn\nn(f)\right|^2\\
	&\le 2m \|f\|_\infty.
\end{align*}
\end{proof}

\begin{rmk}
	Only part \ref{it:first_coord} depends entirely on the Markov
	assumption. Part \ref{it:Q_converges} depends on the weaker
	assumption that the conditional distribution of $\omega_0$
	given $\sigma\omega$ is bounded away from 0.
\end{rmk}

\begin{defn}\label{defn:properness}
	Let $\nu$ be a probability measure on $\Omega\times \proj{d-1}$ with factorisation 
	${\nu(\diff\om, \diff x)=\nu_\om(\diff x)\mP(\diff\om)}$, then $\nu$ is said to be \emph{proper} if for any hyperplane $H$ of $\mathbb{R}^d$, 
	\[
	\nu(\overline{H})=\int_\Omega \nu_\om(\overline{H})\mP(\diff\om)=0.
	\] 
\end{defn}

\begin{prop}\label{prop:proper.LRDS}
Let $(\Omega^+,\mathcal{F}^+,\mathbb{P},\sigma)$ be a one-sided Markovian full shift over $k$
symbols, and $\varphi$ a strongly irreducible and spatially invertible LRDS acting on 
$\mathbb{R}^d$ determined by $\varphi(1,\omega)=M(\omega)$, where $M(\omega)$ depends only on the 
first coordinate $\omega_0$. Let $\psi$ be the NRDS acting on $\proj{d-1}$ induced by $\varphi$. 
Then the invariant measure $\nu$ for $\psi$ is proper.
\end{prop}

\begin{proof}	
Denote by $\Gamma(l)$ the set of $l$-dimensional linear subspaces of $X$.
Let $l_{\min}$ be the smallest such that $\nu(\ovW)>0$ for some $W\in \Gamma(l)$.
Suppose $l_{\min}< d$, and define
\begin{align*}
	r_\omega& := \sup\left\{\nu_\omega(\ovW): W\in\Glm \right\} \text{ for each }
	\omega \in \Omega^+,\\
	r_* &:= \ess\sup r_\omega,\\
	\HH &:= \left\{ (\omega, W) \, : \, \nu_\omega(\ovW) = r_* \right\} \subset \Omega^+ \times \Glm,\\
	l_\omega & :=\inf\{l\le d:\nu_\omega(\ovW)>0, \text{ for some }W\in\Gamma(l)\}.
\end{align*}
We write $\pi_1$ and $\pi_2$ for the projections of $\HH$ onto $\Omega^+$ and
$\Glm$ respectively. The skew product $\SPhi$ maps $\Omega^+\times \Glm$ to itself by
$\SPhi(\omega,W) = (\sigma\omega, M_{\omega_0}W)$.

Fix some $\omega \in \Omega$ with $l_\omega=l_{\min}$. 
There must be a sequence of $l_\omega$-dimensional linear subspaces $V_1,V_2,\dots$
such that $\nu_\omega(V_j) \ge r_\omega - 2^{-j}$. If $V_j\ne V_{j'}$ then $\dim(V_j\cap V_{j'})<l_\omega$, hence (by the minimality of $l_\omega$) $\nu_\omega(V_j\cap V_{j'})=0$. 

For almost every $\omega\in \Omega^+$, and any $W\in \Glm$,
by \cite[Theorem 1.4.5]{Arnold}
\begin{equation} \label{E:invar_factor}
	\nu_{\sigma\om} (\ovW) = \mE\left[ \nu_{\cdot} \bigl(\widehat{M}(\cdot)^{-1}(\ovW) \bigr) 
	\, \bigm| \, \sigma^{-1} \F^+ \right](\om).
\end{equation}
If $(\sigma\omega, W)\in \HH$ then we have
$$
	r_* = r_{\sigma\omega} = \mE\left[ \nu_{\cdot} \bigl(\widehat{M}(\cdot)^{-1}(\ovW) \bigr) 
	\, \bigm| \, \sigma^{-1} \F^+ \right](\om).
$$
Thus, if we associate to any $\sigma\omega\in \pi_1\HH$ a $W_{\sigma\omega}$ such that
$(\sigma\omega,W_{\sigma\omega})\in \HH$,
\begin{align*}
	r_*\mP\left\{ \sigma\omega\in \pi_1(\HH) \right\}
		&= \mE\left[ r_{\sigma\omega} \mathbf{1}\left\{ \sigma\omega\in \pi_1(\HH) \right\} \right]\\
		&= \mE\left[\nu_{\om} \left( \widehat{M}(\omega)^{-1}(\ovW_{\sigma\omega}) \right) 
		  \mathbf{1}\left\{ \sigma\omega\in \pi_1(\HH) \right\}\right].
\end{align*}
Since 
$$
  \nu_{\om} \left( \widehat{M}(\omega)^{-1}(\ovW_{\sigma\omega}) \right)
  	\mathbf{1}\left\{ \sigma\omega\in \pi_1(\HH) \right\}
  	\le r_* \mathbf{1}\left\{ \sigma\omega\in \pi_1(\HH) \right\},
$$
and they have the same expected value, they must be almost surely equal.

We conclude that $\HH$ is almost-surely (in the measure induced by $\nu$)
invariant under $\SPhi$, hence that $\HH_*:= \{W\in \Glm \, :\, \mP\{(\omega,W)\in \HH\}>0\}$ 
is $\mP$-almost surely invariant under
$\cS_\varphi$.

Suppose that for some $r>0$ there were infinitely many $W\in \Glm$ such that
${\mP\{\nu_\omega(\ovW) \ge r \}>0}$.
As $\mP$ is Markovian, Proposition \ref{prop:limit.inv.meas} \ref{it:first_coord}
implies that for fixed $W$, $\nu_\omega(\ovW)$ can take on only $k$ distinct values, with probabilities $q_1,\dots,q_k$. Hence $\mP\{\nu_\omega(\ovW) \ge r\}>0$ implies
${\mP\{\nu_\omega(\ovW) \ge r \}> \min q_i}$, so that $\nu(\ovW) \ge r_* \min q_i$.
Since $\nu(\ovW \cap \ovW')=0$ for $W\ne W'\in \Glm$, and since $\nu$ is a finite measure,
this contradicts the assumption that there are infinitely many. This immediately
shows that $\HH_*$ is finite. But $\HH_*$ also cannot be empty, as that would
imply, by definition of $r_*$
as the essential supremum of $r_\omega$, the existence of an infinite sequence of distinct $W$
with $\mP\{ \nu_\omega(\ovW)\ge r_*/2 \}>0$. Thus $\HH_*$
is finite, nonempty, and invariant under $\cS_\varphi$,
contradicting the strong irreducibility of $\cS_\varphi$, and so proving that
$l_{\min}=d$. 
\end{proof}

\begin{prop}\label{prop:limit.range}
Let $(\Omega^+,\mathcal{F}^+,\mathbb{P},\sigma)$ be a one-sided Markovian full shift over $k$ symbols, $\varphi$  a
strongly irreducible and spatially invertible LRDS 
determined by $\varphi(1,\omega)=M(\omega)$, 
and having index $r$. Assume $M(\omega)$ depends only on the first coordinate $\omega_0$. 

Then for $\mP$-almost all $\om$, there exists an $r$-dimensional subspace $V(\om)$ of $\mathbb{R}^d$ such that it is the range of any limit point of $\varphi(n,\sigma^{-n}\om)/\|\varphi(n,\sigma^{-n}\om)||$. 
For any non-zero $x\in\mathbb{R}^d$, ${\mP\{\om : x\perp V(\om)\}=0}$.
\end{prop}
\begin{proof}
By \autoref{prop:limit.inv.meas} the invariant measure $\tnu$ for $\psi$ with two-sided time may be represented, for $\mP$-
almost every $\omega$, as
\[
	\lim_{n\to\infty}\varphi(n,\sigma^{-n}\om)Q_j\nu_j=\tnu_\om,\quad{1\le j\le k},
\]
where $\nu=\sum_{i=1}^k q_i\nu_i\otimes \delta_{\{\omega:\omega_0=i\}}$ represents an invariant measure for the NRDS $\psi|_{\mathbb{N}\times \op}$ over one-sided time induced by the LRDS $\varphi$, and
$Q_j\in \Lambda_j = \{S M_j\, : \, S\in \cS_\varphi \}$.

Fix $\om$, and let $R(\om)\not = 0$ be a limit point of $\varphi(n,\sigma^{-n}\om)/\|\varphi(n,\sigma^{-n}\om)\|$ of rank $r(\om)$. Then $\vn Q_j\nu_j$ converges to $R(\om)Q_j\nu_j$
along a subsequence. By \autoref{prop:proper.LRDS}, $\nu$ is proper, hence so is $\nu_j$ for each $1\le j\le k$. 
The kernel of $R(\om)Q_j$ is thus $\nu_j$-null for
$\mP$-almost-every $\omega$,
and thus $R(\om)Q_j\nu_j$ is well-defined.
 
Therefore, for each $1\le j\le k$, almost every $\omega$, and any $Q_j$,
\begin{equation}\label{eq:nuom}
R(\om)Q_j\nu_j=\tnu_\om.
\end{equation}
Since $\varphi$ has index $r$ there exists a sequence $S_j^{(n)}/\|S_j^{(n)}\|$ that converges to a rank $r$ matrix $P$ as $n\to\infty$, where $S_j^{(n)}\in\Lambda_j$. Note 
$QS_j^{(n)}\in\Lambda_j$ for $Q\in \cS_\varphi$. We have
\[
	R(\om)Q S_j^{(n)}\nu_j=\tnu_\om. 
\]
Now either $R(\om)Q P=0$ or $R(\om)Q P\nu_j=\tnu_\om$.

If $R(\om)Q P=0$ for all $Q \in \cS_\varphi$, 
define $L:=\operatorname{span}\{Q Px: Q \in\cS_\varphi, x\in\mathbb{R}^d\}$. 
Then $L$ is a nontrivial proper linear subspace of $\ker R(\om)$, 
and $Q(L)=L$ for each $Q\in\cS_\varphi$. This contradicts the strong irreducibility. 

Therefore $R(\om)Q P\nu_j=\tnu_\om$ for some $Q\in\cS_\varphi$. Thus $\supp\tnu_\om\subset\overline{\im(R(\om)Q P)}$. Let \[
V(\om)=\operatorname{span}\{x\in\mathbb{R}^d:\bar{x}\in\supp\tnu_\om\},
\]
then $\dim V(\om)\le r$, since the rank of $R(\om)Q_jP$ is not greater than $r=\operatorname{rank}(P)$.

On the other hand, if $Q=\operatorname{Id}$ in \eqref{eq:nuom}, we have $R(\om)M_j\nu_j=\tnu_\om$. Thus
\[
	1=\tnu_\om(V(\om))=\nu_j\{x:R(\om)M_jx\in V(\om)\}.
\]
Since $\nu_j$ is proper, the linear space $\{x:R(\om)M_jx\in V(\om)\}$
must be $\mathbb{R}^d$. Therefore\\${V(\om)\supset \im(R(\om)M_j)=\im(R(\om))}$, and $\dim V(\om)\ge r(\om)\ge r$.

We may conclude that $\dim V(\om)=r$, and so  $V(\om)=\im(R(\om))$. Finally, 
\begin{align*}
\mP\{\om :x\perp V(\om)\}&=\mP\{\supp\tnu_\om\subset x^{\perp}\}\\
&=\mP\{\tnu_\om(x^\perp)=1\} \\
&= \int \mathbf{1} \{\tnu_\om(x^\perp)=1\} \mP(\diff \om)\\
&\le \int \tnu_\omega(x^\perp)\mathbb{P}(\diff\omega)\\
&= \nu(x^\perp) \quad \text{ (since $\tnu$ is invariant)} \\
&=0 \quad \text{ (since $\nu$ is proper).}
\end{align*}
\end{proof}

\begin{thm}
\label{thm:Furstenberg23.LRDS}
Let $(\Omega^+,\mathcal{F}^+,\mathbb{P},\sigma)$ be a one-sided Markovian full shift over $k$ symbols, $\varphi$ be a strongly irreducible, and spatially invertible LRDS determined by $\varphi(1,\omega)=M(\omega)$, and assume $M(\omega)$ depends only on the first coordinate $\omega_0$. 
If $\mE[\log^+\|M(\omega)\|]<\infty,$
then for a sequence $(x_n)$ of vectors converging to a non-zero $x\in\mathbb{R}^d$ we have 
\[
	\gamma=\lim_{n\to\infty}\dfrac{1}{n}\log\|\varphi(n,\omega)x_n\|,\quad \mathbb{P}\text{-a.s.}
\]
If $\nu$ is the invariant measure for the NRDS $\psi$ induced by $\varphi$, then
\[
	\gamma=\int\log\|M(\omega)\bar{x}\|\nu(\diff\omega,\diff\bar{x}).
\]
Moreover, if  $\mE[\log^+\|M^{-1}(\omega)\|]<\infty$, then
\[
	\dfrac{1}{n}\mathbb{E}\left[\log\|\varphi(n,\omega)x\|\right]\to\gamma\quad (n\to\infty)
\]
uniformly on $\{x\in\mathbb{R}^d:\|x\|=1\}$.
\end{thm}
\begin{proof} 
With the results established above, the proof follows in the same way as in \cite{Bougerol} Chapter III. Corollary 3.4, with \autoref{prop:limit.inv.meas} replacing Chapter II. Lemma 2.1,\autoref{prop:proper.LRDS} replacing Chapter III. Proposition 2.3,  and \autoref{prop:limit.range} replacing Chapter III. Theorem 3.1 in \cite{Bougerol}.
\end{proof}

\begin{prop}\label{prop:contra.ult.LRDS} 
If $(\Omega^+,\mathcal{F}^+,\mathbb{P},\sigma)$ is a one-sided Markovian full shift of finite type, $\varphi$ is a strongly irreducible, contracting and spatially invertible LRDS determined by $\varphi(1,\omega)=M(\omega)$, where $M(\omega)$ depends only on the first coordinate $\omega_0$. Denote
\[\ell(M):=\sup\{\log^+\|M\|,\log^+\|M^{-1}\|\}.\]

Then there exists some $\alpha_0$ such that when $0<\alpha\le\alpha_0$,
\[
	\lim_{n\to\infty}\left\{\sup_{\substack{x,y\in\mathbb{R}P^{d-1}\\x\not= y}}\mathbb{E}\left[\dfrac{d^\alpha(\varphi(n,\omega)\cdot x, \varphi(n,\omega)\cdot y)}{d^\alpha(x,y)}\right]\right\}^{1/n}<1.
\]
\end{prop}
\begin{proof}
The proof follows as in \cite[Proposition V.2.3]{Bougerol}, with \autoref{thm:Furstenberg23.LRDS} replacing their Corollary III.3.4.
\end{proof}

\begin{rmk}
In \cite[Section IV.1]{Bougerol}, it is proved that in the i.i.d. case, under the finite expectation condition, if the matrix set is $p$-strongly irreducible and $p$-contracting (i.e., the matrix set formed by $p$-th exterior power of each matrix is strongly irreducible and contracting resp.), then the $p$-th Lyapunov exponent $\gamma_p$ is strictly greater than the $(p+1)$-th Lyapunov exponent $\gamma_{p+1}$ provided that $\gamma_p\not=-\infty$ ($p\le d-1$). This can be generalised to the Markovian product case with the same proof as that in \cite{Bougerol} and the help of the results in this section.
\end{rmk}

\section{Transfer Operators}
\label{sec:to}
Let $\{M_i\}_{1\le i\le k}$ be a finite set of invertible matrices satisfying strong irreducibility and the contracting property. In this section, we will find a proper function space on which we can define transfer operators and prove the corresponding spectral properties. 

\subsection{Function Space $\gta$}
\label{sec:gta}
Let $\K=\{1,2,\dots, k\}$ and define the one-sided shift space $\op=\K^{\mathbb{N}}$ over $k$ symbols.  Let $0<\theta< 1$, define the metric $\dt$ on $\op$ by $\dt(\om,\tom)=\theta^N$ where $N$ is the largest integer such that $\om_i=\tom_i, 0\le i<N.$ 
Denote the real projective space $\mathbb{R}P^{d-1} \ (d\ge 2)$ by $\X$. Define the metric $d_\X$ on the real projective space by $\dx(x,y)=(1-\langle x,y\rangle^2)^{1/2}$, where $x,y\in\mathbb{R}P^{d-1}$ are two unit vectors.

Let $0<\theta<1,\alpha>0$. Define $\gta$ to be a space 
whose elements are  functions $f:\op \times \X\to \mathbb{C}$ such that for any $f\in \gta$, 
\begin{enumerate}[label=(\alph*)]
\item uniformly for each $x\in \X$, $f(\om,x)$ is Lipschitz with respect to $\om$ under the metric $\dt$; 
\item uniformly for each $\om\in\op$, $f(\om,x)$ is $\alpha$-H\"{o}lder with respect to $x$ under the metric $d_\X$ on $\X$. 
\end{enumerate}

In other words, (a) there exists a finite constant $C_1>0$, not depending on $x$, such that  for each $x\in \X$, we have 
\[|f(\om,x)-f(\tom,x)|\le C_1 \dt(\om,\tom);\]
(b) there exists a finite constant $C_2>0$, not depending on $\om$, such that  for each $\om\in\op$, we have 
\[|f(\om,x_1)-f(\om,x_2)|\le C_2 \dx^\alpha(x_1,x_2).\] 

\vspace{3mm}
Now for any $(\om,x_1),(\tom, x_2)\in\opx$, if $f\in \gta$ we have
\begin{align*}|f(\om,x_1)-f(\tom,x_2)|&\le|f(\tom,x_2)-f(\om,x_2)|+|f(\om,x_1)-f(\om,x_2)|\\
&\le C_1 \dt(\om,\tom)+C_2 d_\X^\alpha(x_1,x_2)\\
&\le\max\{C_1,C_2\}(d_\X^\alpha(x_1,x_2)+\dt(\om,\tom)).\end{align*}
Thus we can define 
\[
|f|_\ta =\sup_{\substack{\om\not =\tom\in\op\\\text{or}\\x_1\not =x_2\in \X}}\dfrac{|f(\om,x_1)-f(\tom,x_2)|}{d_\X^\alpha(x_1,x_2)+\dt(\om,\tom)}.
\]                  
Conversely, any function $f$ on $\opx$ such that $|f|_\ta <\infty$ obviously lies in $\gta.$

\begin{lemma}\label{ineq1} 
	Given $M\in \mat$ define 
	\[\ell(M):=\sup\{\log^+\|M\|,\log^+\|M^{-1}\|\}.\]
	Then 
	\[d_\X(M\cdot x_1,M\cdot x_2)\le e^{4\ell(M)}d_\X(x_1,x_2),\]
	for any $x_1,x_2\in \mathbb{R}P^{d-1}$.
\end{lemma}
\begin{proof} By Lemma III.5.4 in \cite{Bougerol}, we have
\[\left|\log\|\wedge^pM\|\right|\le p \ \ell(M);\]
\begin{equation}\label{eq:bnd.log.Mx}
\left|\log\|\wedge^pM u\|\right|\le p \ \ell(M)\|u\|,
\end{equation}
for any $u\in\wedge^p\mathbb{R}^d$, where $1\le p<d$ is an integer. Therefore
\[\dfrac{d_\X(M\cdot x_1,M\cdot x_2)}{d_\X(x_1,x_2)}=\dfrac{\|Mx_1\wedge Mx_2\|}{\|Mx_1\|\|Mx_2\|}\dfrac{\|x_1\|\|x_2\|}{\|x_1\wedge x_2\|}\le e^{4\ell(M)}.\]
\end{proof}

\begin{prop}\label{L:banach} 
	$(\gta,\|\cdot\|_\ta )$ is a Banach space, where $\|\cdot\|_\ta =|\cdot|_\infty+|\cdot|_\ta $. 
\end{prop}
\begin{proof}
Since a sum of norms is a norm, we need only to show completeness.

Let $\{f_n\}_{n\ge 1}$ be a Cauchy sequence with respect to $\|\cdot\|_\ta $. Then $\{\|f_n\|_\ta:n\ge 1\}$ is bounded. Consequently $\{f_n\}_{n\ge 1}$ is uniformly bounded (since it is bounded under the sup-norm) and equicontinuous (since it is bounded under $|\cdot|_\ta$). As $\opx$ is a compact space, by the Arzel\`{a}--Ascoli Theorem, $\{f_n\}_{n\ge 1}$ has a limit point $f$ in the $|\cdot|_\infty$-topology. We want to show that $f_n$ converges to the same limit in the $|\cdot|_\ta $-topology.

For any $\epsilon>0$, there exists $N>0$, such that when $m,n>N$, $\|f_m-f_n\|_\ta<\epsilon$. So $|f_m-f_n|_\ta<\epsilon$. Fixing two points $(\om,x)$ and $(\tom, y)$,
\begin{align*}
\bigl|(f_m(\om,x)- f_n(\om,x))-(f_m(\tom,y)-f_n(\tom,y) \bigr|
 & \le |f_m-f_n|_\ta \dist\\
&<\epsilon\dist.
\end{align*}
As $f_n$ converges to $f$ uniformly through a subsequence, we have
\[
  \bigl|(f_m(\om,x)- f(\om,x))-(f_m(\tom,y)-f(\tom,y) \bigr|
    \le\epsilon\dist.
\]
So $|f_m-f|_\ta\le\epsilon$ for $m\ge N$. This completes the proof.
\end{proof}

\subsection{Markovian Transfer Operators on $\gta$}
\label{sec:mkv.to}
Given a random matrix $\SM(\omega)$ on the shift space $(\Omega^+,\mathcal{F}^+,\mP)$ over $k$ symbols, we write $p(\omega^{(1)}, \omega^{(0)})$ for the conditional
probability $\mP(\omega = \omega^{(0)} \, | \, \sigma\om = \omega^{(1)} )$,
where $\omega^{(0)}, \omega^{(1)} \in \Omega^+$.
When $\mP$ is Markovian this probability is induced by the $k\times k$ positive stochastic matrix $P$ given by 
\[
	P_{ij}=\mP(\omega_0=i \, | \, \omega_1=j),\quad 1\le i,j\le k, 
\]
and the initial probability vector $\mathbf{p}_0$. Now we want to define a family of parametrised transfer operators $\cL_t$ on $\gta$ by
\begin{equation}\label{eq:defn.L.mkv}
(\cL_tw)(\omega,x)=\sum_\oso p(\om,\om')e^{t\log\|\SM(\om')x\|}w\left(\om', \SM(\om')\cdot x\right).
\end{equation}
If \eqref{eq:defn.L.mkv} yields a well-defined operator on $\gta$ then
the $n$-th power of $\cL_t$ is
\begin{equation}
(\cL_t^nw)(\om,x)=\sum_{\nso} P(n,\omg{n})e^{t\log\|\psi(n,\omg{n})x\|}w(\omg{n}, \psi(n,\omg{n})\cdot x),
\end{equation}
where for any given $\omg{n}\in\op, n\ge 1$, we write
\begin{align}
\omg{k} & :=\sigma^{n-k}\omg{n},\quad \text{for } 1\le k\le n,\quad \om'=\omg{1}, \, \omega=\omg{0}; \notag\\
\label{eq:P.n} P(n,\omg{0}) & := \mP(\omega = \omg{0} \, | \, \sigma^n\om = \omg{n} ) \\
&= p(\sigma^n\omg{0},\sigma^{n-1}\omg{0})p(\sigma^{n-1}\omg{0},\sigma^{n-2}\omg{0}) \cdots p(\sigma\omg{0},\omg{0}); \notag\\
\label{eq:psi.n}\psi(n,\om) & :=\SM(\om)\SM(\sigma\om)\cdots \SM(\sigma^{n-1}\om).
\end{align}

The goal of this section and the next is to prove that when $t$ is sufficiently small, $\cL_t$ possesses a maximal eigenvalue $\beta(t)$
that is real and simple, with the rest of the spectrum lying strictly inside the open ball ${\{z\in\mathbb{C}:|z|<\beta(t)\}}$. This 
means there is a decomposition ${\cL_t=\beta(t)Q(t)+R(t)}$, where $Q(t)$ is a one-dimensional projection and $R(t)$ has spectral radius smaller than $\beta(t)$. Since
\begin{equation}\label{eq:exp}
(\cL_t^n\mathbf{1})(\omega,x)=\mE[e^{t\log\|\psi(n,\omega^{(n)})x\|}],
\end{equation}
then following \cite[Theorem 5.2]{PP} we see that $\beta'(0)$ equals the Lyapunov exponent associated to this problem of Markovian random matrix products.

We begin by defining a general weighted Markovian transfer operator. 
Given a complex weight function $g$ on $\opx$, define for
each $w\in\gta$
\begin{equation}\label{eq:defn.L.g}
(\cL _g w)(\om,x)=\sum_{\sigma\om'=\om}p(\omega,\omega')e^{g(\omega', \SM(\om')\cdot x)}w(\om',\SM(\om')\cdot x).
\end{equation}
$\cL_g$ may be written as
\[
	(\cL_gw)(\om,x)=\sum_{(\om',x')\in \Phi^{-1}(\om,x)}p(\om,\om')e^{g(\om',x')}w(\om',x') ,
\]
where $\Phi$ is the skew product on $\opx$, which here takes the form
\begin{equation}\label{eq:Phi}
	\Phi(\om,x)=(\sigma\om, \SM(\om)^{-1}\cdot x).
\end{equation}
Note here $\SM(\omega)^{-1}=\varphi(1,\omega)=M(\omega)$ in \autoref{defn:skew.prod} with the normalised action. 
This transfer operator is similar to the form used in \cite{PP} to study symbolic dynamics.

Unless otherwise indicated, we will assume $g\in\gta$ with $g=u+\imag v$, where $u, v$ are real functions. If $\cL_u\mathbf{1}=1$ we say $\cL_g$ is \emph{normalized}. When $\SM(\omega)$ only depends on the first coordinate $\omega_0$, we also write $\SM_j=\SM(\omega)$ when $\omega_0=j$ with $1\le j\le k$.

\begin{lemma} \label{lem:normalized_sup_norm}
If $\cL_g$ is normalized then $|\cL_g \, w|_\infty\le |w|_\infty$ for any $w\in \gta$.
\end{lemma}

\begin{proof}
	For any $(\omega,x)$ we have
	\begin{align*}
		\left| \cL_g w(\omega,x)\right| &= 
		  \Bigl| \sum_{(\om',x')\in \Phi^{-1}(\om,x)} 
		   p(\om,\om')e^{g(\om',x')}w(\omega', x')\Bigr|\\
		  &\le \sum_{(\om',x')\in \Phi^{-1}(\om,x)} 
		   p(\om,\om')\bigl| e^{g(\om',x')} \bigr|
		   \cdot \bigl| w(\omega', x') \bigr|\\
		   &\le |w|_\infty \cL_g \mathbf{1} (\omega,x)\\
		   &= |w|_\infty \, .
	\end{align*}
\end{proof}

\begin{thm}\label{thm:Lg.bnd}
Let $\shft$ be a one-sided shift space over $k$ symbols, and \[\SM(\omega):\op\to\mat\] depends only on the first coordinate $\om_0$. Then $\cL_g$ defined by \eqref{eq:defn.L.g} is a bounded operator on $\gtan$.
\end{thm}
\begin{proof}
We note first that if $g\in\gta$, then $e^g\in\gta$. 
This follows from
\[
	\left|e^{g(\omega,x)}-e^{g(\tom,y)}\right|
		\le e^{|g|_\infty}\bigl|g(\om,x)-g(\tom,y)\bigr|. 
\]
We also have $|\cL_g|_\infty\le|e^g|_\infty<\infty$ since
\begin{equation}\label{eq:Lg.unif.bnd}
|\cL_gw|_\infty\le\sum_\oso p(\om,\om')|e^g|_\infty|w|_\infty =|e^g|_\infty|w|_\infty.
\end{equation}

For any $\om, \tom\in\op$ and $x\in\X$, denote by $i\om$ the
sequence defined by $(i\omega)_0=i$ and $\sigma(i\omega)=\omega$. We have
\begin{align*}
\bigl|(\cL_gw)(&\om,x)-(\cL_gw)(\tom,x) \bigr|\\
\le & \sum_{i=1}^k \left|p(\om, i\om)e^{g(\om, \SM_i\cdot x)}w(i\om, \SM_i\cdot x)-p(\tom,i\tom)e^{g(i\tom, \SM_i\cdot x)}w(i\tom, \SM_i\cdot x)\right|\\
\le & \sum_{i=1}^k \left|p(\om, i\om)-p(\tom, i\tom)\right||e^g|_\infty |w|_\infty + \sum_{i=1}^k p(\tom, i\tom)\left|e^{g(i\om, \SM_i\cdot x)} - e^{g(i\tom, \SM_i\cdot x)}\right||w|_\infty  \\
&\qquad +\sum_{i=1}^k p(\tom, i\tom)|e^g|_\infty \left|w(i\om,\SM_i\cdot x)- w(i\tom, \SM_i\cdot x)\right|\\
\le & \left(2|e^g|_\infty|w|_\infty + \theta|e^g|_\ta|w|_\infty + |e^g|_\infty \ \theta |w|_\ta\right) \ \dt(\om, \tom);
\end{align*}
and for any $\om\in\op$ and $x, y\in \X$,
\begin{align*}
\bigl|(\cL_gw)(&\om,x)-(\cL_gw)(\om,y)\bigr| \\
\le& \sum_{\sigma\om'=\om} p(\om,\om') \left(|e^{g(\om', \SM(\om')\cdot x)}-e^{g(\om',  \SM(\om')\cdot y)}||w|_\infty +|w(\SM(\om')\cdot x)-w(\SM(\om')\cdot y)||e^g|_\infty\right)\\
\le&  \sum_{\sigma\om'=\om} p(\om,\om') \left( |e^g|_\ta |w|_\infty +|w|_\ta|e^g|_\infty\right)\dxa(\SM(\om')\cdot x,\SM(\om')\cdot y)\\
\le&  \sum_{\sigma\om'=\om} p(\om,\om') \left(|e^g|_\ta |w|_\infty+|e^g|_\infty |w|_\ta )\right) e^{4\alpha\ell(\SM(\om'))}\dxa(x,y)  \qquad \text{(by \autoref{ineq1})}\\
\le&  e^{4\alpha K}\left(|e^g|_\ta|w|_\infty+|e^g|_\infty|w|_\ta\right)\dxa(x,y),
\end{align*}
where $K:=\max_{1\le i\le k}\ell(\SM_i)<\infty$.
Thus
\begin{equation}\label{eq:Lg.equi}
|\cL_gw|_\ta\le \left[2|e^g|_\infty+ (\theta + e^{4\alpha K})|e^g|_\ta\right] |w|_\infty+ (\theta+e^{4\alpha K})|e^g|_\infty|w|_\ta.
\end{equation}
Therefore $\cL_g$ is a bounded operator on $\gtan$ by \eqref{eq:Lg.unif.bnd} and \eqref{eq:Lg.equi}.
\end{proof}

The parametrised Markovian transfer operators defined by \eqref{eq:defn.L.mkv} correspond to the general 
formula $\cL_g$, with the function
${g(\omega,x)=-t\log\|\SM(\om)^{-1}x\|}$. As
\[
	g(\om', \SM(\om')\cdot x)
	=-t\log\left\|\SM(\om')^{-1}\dfrac{\SM(\om')x}{\|\SM(\om')x\|}\right\|=t\log\|\SM(\om')x\|,
\]
we see that this $g$ is in $\gta$.

The following lemma in \cite{Bougerol} (see Chapter V Lemma 4.2) can be used to validate this choice of $g$.

\begin{lemma}\label{lm:log.bnd}  For $t\in\mathbb{C}$ and $0<\alpha\le 1$, there exists $c_1, c_2>0$ such that for any $M\in \mat$,
\begin{enumerate}[label=(\alph*)]
\item\label{lm.item:log} $\displaystyle\sup_{x\not=y\in \X} \dfrac{\left|\log\|Mx\|-\log\|My\|\right|}{d_\X^\alpha(x,y)}\le c_1 \ \ell(M) e^{2\alpha \ell(M)}.$
\item\label{lm.item:e.log} $\displaystyle\sup_{x\not=y\in \X} \dfrac{\left|e^{t\log\|Mx\|}-e^{t\log\|My\|}\right|}{d_\X^\alpha(x,y)}\le c_2 \ e^{[(1+\alpha)|\real t|+2\alpha]\ell(M)}.$
\end{enumerate}
\end{lemma}

By \autoref{lm:log.bnd} \ref{lm.item:log}, for $\om\in\op$ and $x,y\in\X$,
\[
	|g(\om,x)-g(\om,y)|\le c_1|\real t|K e^{2\alpha K}\dxa(x,y).
\]
And for $\om,\tom\in\op,x\in\X$,
\[
	|g(\om,x)-g(\tom,x)|\le 2|\real t|K\dt(\om,\tom).
\]
These two inequalities show that $|g|_\ta<\infty$, yielding
the following corollary to \autoref{thm:Lg.bnd}.

\begin{cor}\label{cor:Lt.bnd}
Let $\shft$ be a one-sided shift space over $k$ symbols, and \[{\SM(\omega):\op\to\mat}\] depend only on the first coordinate $\om_0$. Then $\cL_t$ defined by \eqref{eq:defn.L.mkv} is a bounded operator on $\gtan$.
\end{cor}

\begin{rmk}
Under the conditions of \autoref{thm:Lg.bnd} $\cL_g$ and $\cL_t$ can also be regarded as bounded operators acting on the space of bounded Borel functions on $\opx$ by \eqref{eq:Lg.unif.bnd}. 
\end{rmk}

\begin{thm}\label{thm:mkv.inv.meas}
Let $\cL_0$ be the parametrised transfer operator $\cL_t$ defined by \eqref{eq:defn.L.mkv} when $t=0$, let $\nu$ be a probability measure on $\opx$ whose marginal on $\Omega^+$ is $\mP$. Then the following are equivalent:
\begin{enumerate}[label=(\alph*)]
\item\label{thm.item:phi.inv} $\nu$ is invariant with respect to $\Phi$;
\item\label{thm.item:L.inv} For any $w\in\gta$, we have 
\begin{equation} \label{eq:L.inv}
\int \cL_0 w(\om,x) \ \nu(\diff\om, \diff x)=\int w(\om,x) \ \nu(\diff\om,\diff x)
\end{equation}
\end{enumerate}
\end{thm}

\begin{proof}	
For any $x\in \X$ and $\mP$-almost every $\omega\in \Omega^+$
\begin{align*}
	\cL_0 w(\sigma\omega,x) &= \sum_{\sigma\om'=\sigma\om} p\left(
	  \sigma \om , \omega' \right) w(\om' , \SM_{\om'} \cdot x)\\
	  &= \mE\left[ w(\cdot,\SM_{\, .\,}\cdot x) \, \bigm| \sigma^{-1} \F \, \right] (\om).
\end{align*}
We then have
\begin{equation} \label{E:cL_identity}
\begin{split}
	\int \cL_0 \, w(\omega,x) \nu(\diff\omega,\diff x) 
	  &= \iint \cL_0\, w(\om,x) \nu_\om (\diff x) \, \mP(\diff \omega) \\
	  &=  \iint \cL_0\, w(\sigma\om,x) \nu_{\sigma\om} (\diff x) \, \mP(\diff \omega) \quad \text{ (by the invariance of $\mP$)}\\
	  	  &= \int  \mE\left[ \int w(\cdot,\SM_{\,. \,} \cdot x) \nu_{\sigma\om} (\diff x)\, \bigm| \sigma^{-1} \F \, \right] (\om) \, \mP(\diff \omega) \\
	  &=\iint  w(\om,\SM_{\om} x) \nu_{\sigma\om} (\diff x) \, \mP(\diff \omega).
\end{split}
\end{equation}

Suppose $\nu$ is $\Phi$-invariant, and consider the quadruple $(\omega,X,\tom,\widetilde{X})$,
where $(\omega,X)$ has distribution $\nu$ and $(\tom, \widetilde{X})=(\sigma\omega, \SM_{\omega}^{-1}\cdot X)$, so that $(\tom, \widetilde{X})$ also has distribution $\nu$. The right-hand side
of \eqref{E:cL_identity} is then just the expected value of
$w(\om , \SM_\om \cdot \widetilde{X} ) = w(\om , X) $, proving assertion \ref{thm.item:L.inv}.

Suppose now that assertion \ref{thm.item:L.inv} holds.
Applying the identity \eqref{E:cL_identity} to the function $w\circ \Phi$ we have
\begin{align*}
\int \cL_0 \, [ w\circ \Phi](\om, x)  \nu(\diff\omega,\diff x)  
		&= \iint  w \circ \Phi \left( \om , \SM_{\om}\cdot x \right) \nu_{\sigma\om} (\diff x) \, \mP(\diff \omega)\\
		&= \iint w\left( \sigma\om, \SM_\om^{-1} \cdot \SM_\om \cdot x \right)
			\nu_{\sigma\om} (\diff x) \, \mP(\diff \omega)\\
		&=  \iint w\left(\om, x \right)
			\nu_{\om} (\diff x) \, \mP(\diff \omega) \quad \text{ (by the invariance of $\mP$)}\\
		&=  \iint w\left(\om, x \right)
			\nu (\diff \omega , \diff x).
\end{align*}
The left-hand
side is equal to $\int w\circ \Phi (\om,x)\nu (\diff\om,\diff x)$,
which is then equal to the final term on the right-hand side,
thus proving assertion \ref{thm.item:phi.inv}, the $\Phi$-invariance of
$\nu$. 
\end{proof}

We conclude this section with $|\cdot|_\ta$-bounds for the power $\cL_g^n$ when $n$ is sufficiently large, \autoref{thm:basic.ineq} and \autoref{cor:basic.ineq}. This type of bound is referred to as a \emph{Lasota--Yorke inequality}, from the original version
formulated in \cite{Lasota-Yorke}. The centrality of such inequalities for proving quasi-compactness (stated below in \autoref{thm:qs-cpct})\footnote{\emph{Quasi-compactness}: a bounded linear operator $T$ acting on a Banach space is quasi-compact if $\|T^n-K\|<1$ for some $n\ge 1$ and some compact linear operator $K\not=0$.} --- and hence the crucial fact (stated below in \autoref{thm:final.thm}) was emphasized by \cite{hennion1993}. Applications to transfer operators can also be found in \cite{walkden2013}.  

We start with a lemma which is an easy corollary of \autoref{prop:contra.ult.LRDS}.
\begin{lemma}\label{lm:contract.av}
Let $\shft$ be a one-sided Markovian shift space over $k$ symbols, and $\SM(\omega):\op\to\mat$ depend only on the first coordinate $\om_0$. Assume the matrix set ${\{ \SM(\om):\om\in\op \}}$ is strongly irreducible and contracting, then there exists $0<\delta<1$ such that when $n$ is sufficiently large, we have
\[\sup_{x\not=y}\mE[\dxa(\psi(n,\omega)\cdot x,\psi(n,\omega)\cdot y)]<\delta^n \dxa(x,y),\]
where $\psi(n,\om)$ is given by \eqref{eq:psi.n}.
\end{lemma}
\begin{rmk}\label{rmk:alpha}
Strong irreducibility and the contracting property are both algebraic properties of a matrix set, which does not depend on the choice of the probability measure on the sample space. However, in the proof of  \autoref{prop:contra.ult.LRDS} the choice of $\alpha_0$ depends on the probability measure. We need $\alpha_0<\beta$ for which $\mE[e^{\beta\ell(\SM(\omega))}]<\infty$. Fortunately, the latter condition follows automatically for any $\beta>0$, as in our setting the set of matrices is finite.
\end{rmk}

\begin{thm}\label{thm:basic.ineq}
Let $\shft$ be a one-sided Markovian shift space over $k$ symbols, and 
$\SM:\op\to\mat$ depend only on the first coordinate $\om_0$. Assume the matrix set $\{ \SM(\om):\om\in\op \}$ is strongly irreducible and contracting, and let $\cL_g$ be a normalized weighted Markovian transfer operator defined by \eqref{eq:defn.L.g} such that the function $g(\cdot, x)$ on $\Omega^+$ depends only on the first coordinate $\om_0$. Write
\[
	G(\omg{n},x):=\sum_{i=1}^n g\left(\omg{i}, \SM(\omg{i-1})\cdots \SM(\om')\cdot x\right),
\]
and assume
\begin{equation}\label{eq:elaborate.cond}
\left|e^{G(\omg{n},x)}-e^{G(\omg{n},y)}\right|\le H(n)\dxa(x,y), 
\end{equation}
for $H(n)$ depending only on $n$. Then there exist constants $\theta_0>0$ and $0<\delta_0<1$, such that when $0<\theta<\theta_0$ and $n$ is sufficiently large, we have
\[|\cL_g^nw|_\ta\le (H(n)+2|e^g|_\infty^n)|w|_\infty+\delta_0^n|w|_\ta.\]
\end{thm}
\begin{proof} 
For $\om,\tom\in\op$ and $x\in\X$, as the function $g(\cdot, x)$ on $\op$ depends only on the first coordinate $\om_0$, we have
\begin{align*}
\bigl| (\cL_g^nw)(&\om,x)-(\cL_g^nw)(\tom,x) \bigr|\\
\le & |e^G|_\infty \cdot \biggl|\sum_{\nso} P(n,\omg{n})w\bigl(\omega^{(n)}, \psi(n,\om^{(n)})\cdot x \bigr)\\
&\hspace*{5cm}-\sum_{\tnso}P \bigl( n,\tom^{(n)})w(\tom^{(n)},\psi(n,\tom^{(n)})\cdot x \bigr)\biggr|.
\end{align*}

Each $\omn$ satisfying $\nso$ is of the form $(i_n\dots i_1\om)$. For each $\omn,\tomn$ of the form $\omn=(i_n\dots i_1\om),\tomn=(i_n\dots i_1\tom)$, by \eqref{eq:P.n} and \eqref{eq:psi.n}, we have
\begin{align*}\psi(n,\omn)&=\psi(n,\tomn), \\
P(n,\omn)&=P(n-1,\omn)p(\om, i_1\om),\\
 P(n,\tomn)&=P(n-1,\omn)p(\tom, i_1\tom) 
\end{align*}
Therefore 
\begin{align}
\bigl| (\cL_g^nw)(&\om,x)-(\cL_g^nw)(\tom,x) \bigr| \nonumber \\
\le & |e^G|_\infty\sum_{i=1}^k\sum_{\sigma^{n-1}\omn=(i\om)}P(n-1,\omn)|p(\om, i\om)-p(\tom, i\tom)||w|_\infty \label{eq:Lg.n.equicont1} \\
& \qquad\qquad +|e^G|_\infty\left|w(\omn, \psi(n,\omn)\cdot x) - w(\tomn, \psi(n,\omn)\cdot x)\right| \label{eq:Lg.n.equicont2} \\
\le & |e^g|_\infty^n (2|w|_\infty + \theta^n|w|_\ta)\dt(\om,\tom),\nonumber 
\end{align}
since $\sum_{\sigma^{n-1}\omg{n}=(i\om)} P(n-1, \om^{(n)})=1$,
and by the Markov property of $\mP$ 
$$
\sum_{i=1}^k |p(\om, i\om)-p(\tom, i\tom)|\le 2\cdot \mathbf{1}_{\{ \omega_0\neq \tom_0 \}} \le 2d_\theta(\om,\tom).
$$

Note that $\cL_g$ being normalized implies that 
\[
	1=\cL_g\mathbf{1}=\sum_\oso p(\om,\om')e^{g(\om',\SM(\om')\cdot x)},
\]
hence by iteration $\left|P(n,\omg{n})e^{G(\omg{n},x)}\right|$
defines a probability measure on $\omg{n}$; we write
$\mE_{pe^g}$ for expectation with respect to this distribution.
Now fix any $\om\in\op$ and $x,y\in\X$. We have
\begin{align*}
\bigl|(\cL_g^nw)(&\om,x)-(\cL_g^nw)(\om,y) \bigr|\\
\le & \sum_\nso P(n,\omg{n})\left|e^{G(\omg{n},x)}-e^{G(\omg{n},y)}\right||w|_\infty\\
&\qquad+ \sum_{\nso}\left|w(\omg{n},\psi(n,\omg{n})\cdot x)-w(\omg{n},\psi(n,\omg{n})\cdot y)\right|\\
&\hspace*{5cm} \times \left|P(n,\omg{n})e^{G(\omg{n},x)}\right|\\
\le &  H(n)|w|_\infty \dxa(x,y) +|w|_\ta\mE_{pe^g}\left[\dxa\left(\psi(n,\omg{n})\cdot x,\psi(n,\omg{n})\cdot y\right)\right]\\
\le & \left( H(n)|w|_\infty+\delta^n |w|_\ta \right)\dxa(x,y),
\end{align*}
by \autoref{lm:contract.av} and \autoref{rmk:alpha}.
Therefore, when $n$ is sufficiently large we would have
\[
	|\cL_g^nw|_\ta\le (H(n)+2|e^g|_\infty^n)|w|_\infty+(\theta^n|e^g|_\infty^n+\delta^n)|w|_\ta.
\]
Now choosing $\theta< \theta_0:=\delta/|e^g|_\infty$, we have
\[
	\theta^n|e^g|_\infty^n+\delta^n<2\delta^n<\delta_0^n
\]
for some $0<\delta_0<1$, when $n$ is sufficiently large.
\end{proof}

\begin{cor}\label{cor:basic.ineq}
Let $\cL_g$ be a normalized Markovian transfer operator satisfying the assumptions of \autoref{thm:basic.ineq}, then there exists $\theta_0>0$ such that when $0<\theta<\theta_0$, there exist constants
$C>0$ and $\delta\in (0,1)$ 
such that for all $n$ sufficiently large
\begin{equation} \label{E:basic.ineq}
	|\cL_g^{n}w|_\ta\le C|w|_\infty+\delta^{n}|w|_\ta, 
\end{equation}
for any $w\in\gta$.
\end{cor}

\begin{proof}
By \autoref{thm:basic.ineq}, there exists $m_0 \ge 1$ and $\delta\in (0,1)$ such that for all $w\in\gta$ and $0\le m\le m_0-1$
\[
	|\cL_g^{m_0+m} w|_\ta\le C_{1}|w|_\infty+\delta^{m_0+m}|w|_\ta
\]
where $C_{1}=\max_{m_0 \le n\le 2m_0-1} H(n)+2|e^g|_\infty^{n}$. 
Combining this with \autoref{lem:normalized_sup_norm} we have for $k\ge 1$
$$
	|\cL_g^{(k+1)m_0+m}w|_\ta \le C_{1}|\cL_g^{km_0}w|_\infty
		+\delta^{m_0+m} |\cL_g^{km_0}w|_\ta
$$
By induction we show that for all positive integers $k$
$$
	|\cL_g^{km_0+m}w|_\ta\le C_1\left( 1+\delta^{m_0} +\cdots
		+ \delta^{(k-1)m_0 } \right) + \delta^{km_0} |w|_\ta.
$$
Thus \eqref{E:basic.ineq} holds for $n\ge m_0$ with
$C=C_1/(1-\delta^{m_0})$.
\end{proof}

\section{Spectral Properties}\label{sec:spec}
We start with the spectral properties of $\cL_t$ when $t=0$. Here we define an operator $\mathcal{Q}$ on $\gta$ by 
\begin{equation}\label{eq:defn.Q}
(\mathcal{Q}w)(\omega,x)=\int w \ \diff\nu,
\end{equation}
where $\nu$ is the invariant measure for the RDS $\varphi$ characterised by $\varphi(1,\om)=\SM(\om)^{-1}$. By \autoref{thm:mkv.inv.meas}, we know $\nu$ is also an eigenmeasure of $\cL_0$.
\begin{thm}\label{thm:qs-cpct}
Let $\shft$ be a one-sided Markovian shift space over $k$ symbols, and $\SM(\omega):\op\to\mat$ depend only on the first coordinate $\om_0$. Assume the matrix set ${\{ \SM(\om):\om\in\op \}}$ is strongly irreducible and contracting. Let $\cL_0$ be the parametrised Markovian transfer operator $\cL_t$ defined by \eqref{eq:defn.L.mkv} acting on $\gta$ when $t=0$ and $\mathcal{Q}$ be defined as in \eqref{eq:defn.Q}. Then there exists $\theta_0>0$ such that when $0<\theta<\theta_0$ and $n$ sufficiently large, we have  
\[\|\cL_0^{n}-\mathcal{Q}\|_\ta^{1/{n}}<1.\]
\end{thm}
\begin{proof}
Note $\cL_0$ is the same as the weighted transfer operator $\cL_g$
when $g=0$. Therefore, the condition \eqref{eq:elaborate.cond} is automatically satisfied, with $H(n)=0$. Fix $w\in\gta$, then by \autoref{thm:basic.ineq} we know there exists $\theta_0>0$ and $0<\delta<1$ such that when $n$ is sufficiently large, we have 
\begin{equation}\label{eq:2}
|\cL_0^{n}w|_\ta\le 2 |w|_\infty+\delta^{n}|w|_\ta.
\end{equation}
This shows that $\{\cL_0^{n}w\}$ is equicontinuous. Therefore, by the Arzel\`{a}--Ascoli Theorem, ${\{\cL_0^n w: n\ge 1\}}$ is relatively
compact with respect to the uniform topology. 
Let $\{\cL_0^{n_j}w:j\ge 1 \}$ be a subsequence converging
to a point $w^*$.

For $\om,\tom\in\op, x,y\in \X$ with $\om_0=\tom_0$, by the same proof as for \autoref{thm:basic.ineq}, we have
\[ 
	\bigl|\cL_0^n(\om, x)-\cL_0^n(\tom, y) \bigr|
	\le \theta^n|w|_\ta d_\theta(\om,\tom)+\delta^n |w|_\ta \dxa(x,y).
\]
(Here the term corresponding to \eqref{eq:Lg.n.equicont1} vanishes, because $\om_0=\tom_0$.) Letting $j\to\infty$,
hence $n_j\to\infty$ shows that $w^*(\om, x)$ depends only on $\om_0$. Denote $w^*(\om, x)=w^*_{\om_0}$.

As $\cL_0$ is normalized, by \autoref{lem:normalized_sup_norm}
\[
	\sup w\ge \sup \cL_0 w\ge \cdots \ge \sup \cL_0^n w\ge \cdots .
\]
Thus
$$
  0\ge \sup \cL_0^{n_j} w - \sup \cL_0^{n_j+1} w \to \sup w^* - \sup \cL_0 w^*
$$
as $n\to\infty$ by the uniform convergence to $w^*$, since 
$\cL_0$ is a continuous operator with respect to the sup norm. Similarly,
$$
  0\le \sup \cL_0^{n_{j+1}} w - \sup \cL_0^{n_j+1} w \to \sup w^* - \sup \cL_0 w^*.
$$
We may conclude that $\sup w^*=\sup \cL_0w^*$. Since $\opx$ is compact, the supremum of $w^*$ is attained at a point $(\om, x)\in \opx$, 
and that of $\sup\cL_0w^*$ at a point $(\tom, y)$. We have
\[
	w^*(\om, x)=(\cL_0 w^*)(\tom, y)=\sum_{i=1}^k p(\tom, i\tom)w^*(i\tom, \SM_i\cdot y).
\]
As $w^*$ depends only on $\om_0$, we have 
\[
	\sup w^*=\sum_{i=1}^k p(\tom, i\tom)w_i^*\le \sup w^*.
\]
This shows that $w_i^*=\sup w^*$ for each $1\le i\le k$, and $w^*$ is constant on $\opx$.

As $\nu$ is an invariant measure we have $\mathcal{Q}\cL_0=\mathcal{Q}$. It follows that
\[
	w^*=\mathcal{Q}w^*=\lim_{j\to\infty}\mathcal{Q}\cL_0^{n_j}w=\mathcal{Q}w.
\]
By applying the same arguments to each subsequence of $\{\cL_0^n w:n\ge1\}$, we can obtain a further subsequence converging to the same limit $\mathcal{Q}w$. Thus $\cL_0^n w$ converges uniformly to $\mathcal{Q}w$ as $n\to\infty$.

Now for $n, m\ge 0$, by \eqref{eq:2} we have
\begin{align*}
|(\cL_0^{n+m}-&\mathcal{Q})(w)|_\ta= |\cL_0^n(\cL_0^m-\mathcal{Q})(w)|_\ta\\
\le & 2|(\cL_0^m-\mathcal{Q})(w)|_\infty +\delta^n |(\cL_0^m-\mathcal{Q})(w)|_\ta\\
= & 2 |(\cL_0^m-\mathcal{Q})(w)|_\infty + \delta^n |\cL_0^m w|_\ta \numberthis \label{eq:lminusq}\\
\le & 2 |(\cL_0^m-\mathcal{Q})(w)|_\infty + 2\delta^n |w|_\infty+\delta^{n+m}|w|_\ta, \numberthis \label{eq:nplusk}
\end{align*}
where the equality in \eqref{eq:lminusq} follows from
\[
\left|(\cL_0^nw-\mathcal{Q}w)(\om,x)-(\cL_0^nw-\mathcal{Q}w)(\tom,y)\right|=|(\cL_0^nw)(\om,x)-(\cL_0^nw)(\tom,y)|. 
\]

Define $\Gamma=\{w\in\gta: \|w\|_\ta\le 1 \}$, which is compact in the uniform topology by the Arzel\`{a}--Ascoli theorem. Therefore for any $\epsilon>0$, there exists an integer $T$ (not depending on $w\in\Gamma$), such that when $m>T$, we have $|\cL_0^mw-\mathcal{Q}w|_\infty<\epsilon$. Choose an integer $N>0$ such that when $n>N$, $\delta^n <\epsilon$. Then for any $w\in\Gamma$, by \eqref{eq:nplusk}, we have
\begin{align*}
\|(\cL_0^{n+m}-\mathcal{Q})(w)\|_\ta = & |(\cL_0^{n+m}-\mathcal{Q})(w)|_\infty+|(\cL_0^{n+m}-\mathcal{Q})(w)|_\ta \\
\le & \epsilon+(2\epsilon + 2\epsilon|w|_\infty+\epsilon|w|_\ta)\le 6\epsilon.
\end{align*}
Now the conclusion follows by choosing $\epsilon<1/6$.

\end{proof}
\begin{rmk}
$\cL_0$ obviously has 1 as an eigenvalue. This theorem tells us that $1$ is, in fact, the top eigenvalue, and that it is simple and isolated. More precisely, $\cL_0$ can be written as $\mathcal{Q}+\mathcal{R}$ where $\mathcal{Q}$ is the one-dimensional projection given by \eqref{eq:defn.Q} and $\mathcal{R}=\cL_0-\mathcal{Q}$ has spectral radius strictly smaller than $1$ with $\mathcal{QR}=\mathcal{RQ}=0$, since the spectral radius of $\mathcal{R}$ is given by $\inf\{\|\mathcal{R}^n\|_\ta^{1/n}:n\ge 1 \}$ and $\mathcal{R}^n=\cL_0^n-\mathcal{Q}$.
\end{rmk}

We now turn to the corresponding Lasota--Yorke inequality for $\cL_t$.

\begin{lemma}\label{lm:nK}
For any $n\ge 1,\omega\in\op$, we have $\ell(\psi(n,\om))\le nK$,  where $K:=\max_{1\le i\le k}\ell(\SM_i).$
\end{lemma}
\begin{proof} 
We notice that for any $A, B\in\mat$, $\ell(AB)\le\ell(A)+\ell(B)$ follows immediately from the following two inequalities.
\[\log^+\|AB\|\le \log^+\|A\|+\log^+\|B\|\le \ell(A)+\ell(B),\]
\[\log^+\|(AB)^{-1}\|\le \log^+\|B^{-1}\|+\log^+\|A^{-1}\|\le \ell(A)+\ell(B).\]
\end{proof}

\begin{thm}\label{thm:basic.ineq.Lt}
Let $\shft$ be a one-sided Markovian shift space over $k$ symbols, and $\SM(\omega):\op\to\mat$ depend only on the first coordinate $\om_0$. Assume the matrix set $\{ \SM(\om):\om\in\op \}$ is strongly irreducible and contracting. Let $t$ be purely imaginary and $\cL_t$ be the parametrised Makovian transfer operators defined by \eqref{eq:defn.L.mkv} acting on $\gta$. Then there exists $\theta_0>0$ such that for any $\theta \in (0,\theta_0)$ there exist constants
$C>0$ and $\delta\in (0,1)$ 
\[
	|\cL_t^{n_k}w|_\ta\le C|w|_\infty+\delta_0^{n_k}|w|_\ta
\]
for any $w\in\gta$.
\end{thm}
\begin{proof}
Notice first that when $t$ is purely imaginary, $\cL_t$ is normalized. Now by \autoref{cor:basic.ineq}, we only need to verify inequality \eqref{eq:elaborate.cond}. Note $\cL_t$ equals $\cL_g$ when 
\[g(\om, x)=-t\log\|\SM(\om)^{-1}x\|\]
by the discussion after \autoref{thm:Lg.bnd}. Since $M(\om)$ only depends on the first coordinate $\om_0$, we know that $g$ (as a function of $\Omega^+$) also only depends on the first coordinate. Moreover,
\[G(\omg{n},x)=t\log\|\SM(\omg{n})\cdots \SM(\om')x\|. \]
Then by \autoref{lm:log.bnd} \ref{lm.item:e.log}, we have
\begin{align*}
|e^{G(\omg{n},x)}-e^{G(\omg{n},y)}|&\le c_2 e^{[(1+\alpha)|\real t|+2\alpha]\ell(\psi(n,\omg{n}))}\dxa(x,y)\\
&\le c_2 e^{[(1+\alpha)|\real t|+2\alpha]nK}\dxa(x,y),
\end{align*}
by \autoref{lm:nK}. Therefore $\eqref{eq:elaborate.cond}$ is satisfied with $H(n)=c_2 e^{2\alpha nK}$ as $\real t=0$.
\end{proof}

Now we now state a general perturbation theorem from \cite{Keller},
from which our key result \autoref{thm:final.thm} will follow.
\begin{thm}\label{thm:perturbation}
Let $(B,\|\cdot\|)$ be a Banach space with a second norm $|\cdot|$ such that $|\cdot|\le\|\cdot\|$ (we do not require $(B,|\cdot|)$ to be complete). Let $\cP_\epsilon:B\to B$ be a family of bounded linear operators for some parameter $\epsilon$ in a set containing $0$. Suppose
\begin{enumerate}[label=(\alph*)]
\item the inclusion $\iota:(B,\|\cdot\|)\hookrightarrow(B,|\cdot|)$ is compact;
\item there exists a set $\mathcal{E}$ containing $0$, for which $\cP_\epsilon \ (\epsilon\in\mathcal{E})$ satisfies the uniform Lasota--Yorke inequality. That is, there exist $n_0>0$, $0<\delta<1$, and $C>0$ such that 
\[\|\cP_\epsilon^{n_0}w\|\le \delta\|w\|+C|w|,\quad \forall\epsilon\in\mathcal{E};\]
\item there is a monotone upper semi-continuous function $\psi$ such that $\psi(\epsilon)\to 0$ as $\epsilon\to 0$ and $\sup_{w\in B}|(\cP_\epsilon-\cP_0) w|/\|w\|\le \psi(\epsilon)$.
\end{enumerate}
Suppose $\cP_0$ has a simple maximal eigenvalue $1$ with corresponding eigenprojection $\nu$. Then there exists $\epsilon_0>0$ such that when $|\epsilon|<\epsilon_0$ in $\mathcal{E}$, $\cP_\epsilon$ has a simple maximal eigenvalue with a corresponding eigenprojection $\nu_\epsilon$.
\end{thm}

\begin{thm}\label{thm:final.thm}
Let $\cL_t$ be the parametrised transfer operator on $\gta$ defined as in \eqref{eq:defn.L.mkv} with $t$ purely imaginary. Then $\cL_t$ can be decomposed as
\[\cL_t=\beta(t)\mathcal{Q}(t)+\mathcal{R}(t),\]
where $\mathcal{Q}(t)$ is a one-dimensional projection, $\mathcal{R}(t)$ has the spectral radius strictly smaller than $\beta(t)$ with $\mathcal{Q}(t)\mathcal{R}(t)=\mathcal{R}(t)\mathcal{Q}(t)=0$. 
At $t=0$ we have $\beta(0)=1, \mathcal{Q}(0)=\mathcal{Q}$ given by \eqref{eq:defn.Q}.
\end{thm}
\begin{proof}
We apply \autoref{thm:perturbation} with $|\cdot|=|\cdot|_\infty$ and $\|\cdot\|=\|\cdot\|_\ta$. We need to verify the theorem's three conditions.

For (a), take $\{w_n\}$ to be a bounded sequence in $\gtan$ with $\|w\|_\ta\le 1$. They are uniformly bounded (since $|w_n|_\infty\le 1$) and equicontinuous (since $|w_n|_\ta\le 1$). By the Arzel\`{a}--Ascoli Theorem $\{w_n\}$ is relatively compact in the uniform topology. 

(b) is satisfied because of \autoref{thm:basic.ineq.Lt}. 

Since $|e^{\imag h}-1|^2\le h^2$ for any $h\in\mathbb{R}$ when $t$ is purely imaginary,
\[
  |\cL_tw-\cL_0w|\le |e^t -1|K|w|_\infty\le |t|K\|w\|_\ta.
\]
Thus (c) holds with $\psi(t)=|t|K$.

The fact that $\cL_0$ has a simple maximal eigenvalue $1$ is guaranteed by \autoref{thm:qs-cpct}. Therefore the results follows by applying \autoref{thm:perturbation}.
\end{proof}

\begin{rmk}
As before, by \eqref{eq:exp} and the general perturbation theory of \cite{D.Sch} Section VII.6, we can deduce that $\beta(t)$ is analytic and that $\beta'(0)$ equals the top Lyapunov exponent associated to this Markoviansystem of matrix products.
\end{rmk}

\bibliographystyle{plain}
\bibliography{random_matrices,operator_algebras}
\end{document}